\documentclass[12pt,draftcls,onecolumn]{IEEEtran}
\ifCLASSINFOpdf
\else
\fi
\usepackage{graphics} 
\usepackage{epsfig} 
\usepackage{amsmath} 
\usepackage{amssymb}  
\DeclareMathOperator{\diag}{diag}
\newtheorem{lemma}{Lemma}[section]
\newtheorem{theorem}[lemma]{Theorem}
\newtheorem{proposition}[lemma]{Proposition}
\newtheorem{corollary}[lemma]{Corollary}
\newtheorem{definition}[lemma]{Definition}
\newtheorem{example}[lemma]{Example}
\newtheorem{remark}[lemma]{Remark}
\newcommand{\R}{\ensuremath{\mathbb{R}}}

\newcommand{\id}{\ensuremath{\mbox{id}}}

\newcommand{\K}{\ensuremath{\mathcal{K}}}

\newcommand{\Kinf}{\ensuremath{\mathcal{K}_\infty}}

\newcounter{enumctr}
\newenvironment{enum}{\begin{list}{(\roman{enumctr})}
{\usecounter{enumctr}}}{\end{list}}

\begin{document}
\title{A small gain condition for interconnections of ISS systems with mixed ISS characterizations}

\author{Sergey~Dashkovskiy, Michael~Kosmykov, Fabian~Wirth
\thanks{S. Dashkovskiy is with Faculty of Mathematics and Computer Science, University of Bremen, 28334 Bremen, Germany
        {\tt\small dsn@math.uni-bremen.de}}%
\thanks{M. Kosmykov is with Faculty of Mathematics and Computer Science, University of Bremen, 28334 Bremen, Germany
        {\tt\small kosmykov@math.uni-bremen.de}}%
\thanks{F. Wirth is with Institute for Mathematics, University of W\"urzburg, 97074 W\"urzburg, Germany
        {\tt\small wirth@mathematik.uni-wuerzburg.de}}%
}



\maketitle

\begin{abstract}
  We consider interconnected nonlinear systems with external inputs, where
  each of the subsystems is assumed to be input-to-state stable (ISS).
  Sufficient conditions of small gain type are provided guaranteeing that the
  interconnection is ISS with respect to the external input. To this end we
  extend recently obtained small gain theorems to a more general type of
  interconnections. The small gain theorem provided here is applicable to
  situations where the ISS conditions are formulated differently for each
  subsystem and are either given in the maximization or the summation
  sense. Furthermore it is shown that the conditions are compatible in the
  sense that it is always possible to transform sum formulations to maximum
  formulations without destroying a given small gain condition. An example
  shows the advantages of our results in comparison with the known ones.
\end{abstract}


%
\IEEEpeerreviewmaketitle

\section{Introduction}
Stability of nonlinear systems with inputs can be described in different ways
as for example in sense of dissipativity \cite{WiJ72}, passivity \cite{Vid81},
\cite{WeA04}, input-to-state stability (ISS) \cite{Son89} and others. In this
paper we consider general interconnections of nonlinear systems and assume
that each subsystem satisfies an ISS property.  The main question of the paper
is whether an interconnection of several ISS systems is again ISS. As the ISS
property can be defined in several equivalent ways we are interested in
finding optimal formulations of the small gain condition that are adapted to a
particular formulation. In particular we are interested in a possibly sharp
stability condition for the case when the ISS characterization of single
systems are different. Moreover we will provide a construction of an ISS
Lyapunov function for interconnections of such systems.

Starting with the pioneering works \cite{JTP94}, \cite{JMW96}
stability of interconnections of ISS systems has been studied by
many authors, see for example \cite{LaN03}, \cite{AnA07},
\cite{Cha05}, \cite{Ito08}. In particular it is known that
cascades of ISS systems are ISS, while a feedback interconnection
of two ISS systems is in general unstable.  The first result of
the small gain type was proved in \cite{JTP94} for a feedback
interconnection of two ISS systems.  The Lyapunov version of this
result is given in \cite{JMW96}. Here we would like to note the
difference between the small gain conditions in these papers. One
of them states in \cite{JMW96} that the composition of both gains
should be less then identity. The second condition in \cite{JTP94}
is similar but it involves the composition of both gains and
further functions of the form $(\id+\alpha_i)$. This difference is
due to the use of different definitions of ISS in both papers.
Both definitions are equivalent but the gains enter as a maximum
in the first definition, and a sum of the gains is taken in the
second one. The results of \cite{JTP94} and \cite{JMW96} were
generalized for an interconnection of $n\ge2$ systems in
\cite{DRW07-mcss,DRW07-NOLCOS}, \cite{JiW08}, \cite{KaJ09}. In
\cite{DRW07-mcss,DRW07-NOLCOS} it was pointed out that a
difference in the small gain conditions remains, i.e., if the
gains of different inputs enter as a maximum of gains in the ISS
definition or a sum of them is taken in the definition.  Moreover,
it was shown that the auxiliary functions $(\id+\alpha_i)$ are
essential in the summation case and cannot be omitted,
\cite{DRW07-mcss}. In the pure maximization case the small gain
condition may also be expressed as a condition on the cycles in
the gain matrix, see e.g.
\cite{Tee05,DRW07-mcss,Rueff07,JiW08,KaJ09}. A formulation of ISS
in terms of monotone aggregation functions for the case of many
inputs was introduced in \cite{Rueff07,DRW-CDC08,DRW09}. For
recent results on the small gain conditions for a wider class of
interconnections we refer to \cite{JiW08}, \cite{ItJ07},
 \cite{KaJ09}.
 In \cite{ItJ09} the authors consider necessary and sufficient small gain
 conditions for interconnections of two ISS systems in dissipative form.

 In some applications it may happen that the gains of a part of systems of an
 interconnection are given in maximization terms while the gains of another
 part are given in a summation formulation. In this case we speak of mixed ISS
 formulations. We pose the question whether and where we need the functions
 $(\id+\alpha_i)$ in the small gain condition to assure stability in this
 case. In this paper we consider this case and answer this question. Namely we
 consider $n$ interconnected ISS systems, such that in the ISS definition of
 some $k\le n$ systems the gains enter additively.  For the remaining systems
 the definition with maximum is used.  Our result contains the known small
 gain conditions from \cite{DRW07-mcss} as a special case $k=0$ or $k=n$,
 i.e., if only one type of ISS definition is assumed. An example given in this
 paper shows the advantages of our results in comparison with the known ones.

 This paper is organized as follows. In Section~\ref{sec:Preliminaries} we
 present the necessary notation and definitions. Section~\ref{sec:Auxilary}
 discusses properties of gain operators in the case of mixed ISS
 formulations. In particular we show that the mixed formulation can in
 principle always be reduced to the maximum formulation. A new small gain
 condition adapted to the mixed ISS formulation ensuring stability of the
 considered interconnection is proved in Section~\ref{sec:SGT}.
 Section~\ref{sec:Lyapunov} provides a construction of ISS Lyapunov functions
 under mixed small gain conditions. We note some concluding remarks in
 Section~\ref{sec:Conclusion}.

\section{Preliminaries and problem statement}
\label{sec:Preliminaries}

\subsection{Notation}
In the following we set ${\mathbb{R}}_{+}:=[0,\infty)$ and denote the
positive orthant ${\mathbb{R}}_{+}^n:= [0,\infty)^n$. The transpose of a
vector $x \in \mathbb{R}^n$ is denoted by $x^T$.  On ${\mathbb{R}}^n$ we
use the standard partial order induced by the positive orthant given by
\[
\begin{array}{l}
x\geq y\, \Longleftrightarrow\, x_i \geq y_i,\quad i=1,\ldots,n,\\
x>y\,\Longleftrightarrow\, x_i>y_i,\quad i=1,\ldots,n.
\end{array}\]
With this notation ${\mathbb{R}}_{+}^n:= \{ x \in \R^n \;:\; x
\geq 0 \}$. We write $x\not\geq y\,\Longleftrightarrow\, \exists
\,i\in\{1,\ldots,n\}: \, x_i<y_i.$ For a nonempty index set
$I\subset\{1,\ldots,n\}$ we denote by $|I|$ the number of elements
of $I$. We write $y_I$ for the restriction $y_I:=(y_i)_{i\in I}$
of vectors $y\in \R_{+}^n$. Let $R_I$ be the anti-projection
${\mathbb{R}}_{+}^{|I|}\rightarrow{\mathbb{R}}_{+}^{n}$, defined
by
$$x\mapsto\sum\limits_{k=1}^{|I|}x_ke_{i_k},$$
where $\{e_k\}_{k=1,\dots,n}$ denotes the standard basis in
${\mathbb{R}}^{n}$ and $I=\{i_1,\ldots,i_{|I|}\}$.

For a function $v:{\mathbb{R}}_{+}\mapsto{\mathbb{R}}^{m}$ we
define its restriction to the interval $[s_1,s_2]$ by
\[v_{[s_1,\, s_2]}(t)=\left\{\begin{array}{ll}
v(t),& \mbox{if }t \in [s_1,s_2],\\
0,& \mbox{otherwise.}
\end{array}\right.\]

A function $\gamma:{\mathbb{R}}_{+}\mapsto{\mathbb{R}}_{+}$ is said to be of
class $\cal{K}$ if it is continuous, strictly increasing and $\gamma(0)=0$. It
is of class $\cal{K}_{\infty}$ if, in addition, it is unbounded. Note that for
any $\alpha \in \cal{K}_{\infty}$ its inverse function $\alpha^{-1}$ always
exists and $\alpha^{-1}\in \cal{K}_{\infty}$.  A function $\beta:
{\mathbb{R}}_{+}\times{\mathbb{R}}_{+}\mapsto{\mathbb{R}}_{+}$ is said to be
of class $\cal{KL}$ if, for each fixed $t$, the function $\beta(\cdot,t)$ is
of class $\cal{K}$ and, for each fixed $s$, the function $t\mapsto
\beta(s,t)$ is non-increasing and tends to zero for $t \to \infty$. By
$\id$ we denote the identity map.

Let $|\cdot|$ denote some norm in ${\mathbb{R}}^{n}$, and let in particular
${|x|}_{\max}=\max\limits_i|x_i|$ be the maximum norm. The essential supremum
norm of a measurable function $\phi:\R_+\to \R^m$ is denoted by
${\|\phi\|}_{\infty}$. $L_{\infty}$ is the set of measurable functions for
which this norm is finite.

\subsection{Problem statement}
Consider the system
\begin{equation}\label{whole system}
\dot{x}=f(x,u),\quad x \in {\mathbb{R}}^{n}, \quad
u\in{\mathbb{R}}^{m},
\end{equation}
and assume it  is forward complete, i.e., for all
initial values $x(0)\in {\mathbb{R}}^{n}$ and all essentially bounded
measurable inputs $u$ solutions $x(t) = x(t;x(0),u)$
exist for all positive times. Assume also that
for any initial value $x(0)$ and input $u$ the solution is unique.

The following notions of stability are used in the remainder of the paper.

\begin{definition}
System (\ref{whole system}) is called
\begin{enum}
\item  {\it input-to-state stable}
(ISS), if there exist functions $\beta\in\cal{KL}$ and
$\gamma\in\cal{K}$, such that
\begin{equation}\label{iss_sum}
|x(t)|\leq\beta(|x(0)|,t)+\gamma({\|u\|}_{\infty})\,,\quad \forall
x(0)\in {\mathbb{R}}^{n}\,, u \in
L_{\infty}(\R_{+},\R^{m})\,,t\ge0.
\end{equation}
\item   {\it globally stable} (GS), if there
  exist functions $\sigma$, $\hat{\gamma}$ of class $\cal{K}$, such that
\begin{equation}\label{gs_sum}
|x(t)|\leq\sigma(|x(0)|)+\hat{\gamma}({\|u\|}_{\infty}) \,,\quad \forall
x(0)\in {\mathbb{R}}^{n}\,, u \in L_{\infty}(\R_{+},\R^{m})\,,
t\geq 0.
\end{equation}

\item System (\ref{whole system}) has the {\it asymptotic gain} (AG)
  property, if there exists a function $\overline{\gamma}\in \cal{K}$, such that
\begin{equation}\label{ag_sum}\limsup\limits_{t \to
    \infty}|x(t)|\leq\overline{\gamma}({\|u\|}_{\infty})\,,\quad \forall
x(0)\in {\mathbb{R}}^{n}\,, u \in L_{\infty}(\R_{+},\R^{m})\,.
\end{equation}
\end{enum}
\end{definition}

\begin{remark}
An equivalent definition of ISS is obtained if
instead of using summation of terms in
(\ref{iss_sum}) the maximum is used as follows:
\begin{equation}\label{iss_max}
|x(t)|\leq\max\{\tilde\beta(|x(0)|,t),\tilde\gamma({\|u\|}_{\infty})\}.
\end{equation}
Note that for a given system sum and maximum formulations may lead to
different comparison functions $\tilde\beta,\,\tilde\gamma$ in
(\ref{iss_max}) than those in (\ref{iss_sum}). In a similar manner an
equivalent definition can be formulated for GS in maximization terms.
\end{remark}
\begin{remark}
  In \cite{SoW96} it was shown that a system \eqref{whole system} is ISS
  if and only if it is GS and has the AG property.
\end{remark}

We wish to consider criteria for ISS of interconnected systems. Thus
consider $n$ interconnected control systems given by
\begin{equation}\label{subsystems}
\begin{array}{ccc}
{\dot{x}}_1&=&f_1(x_1,\ldots,x_n,u_1)\\
&\vdots&\\
{\dot{x}}_n&=&f_n(x_1,\ldots,x_n,u_n)
\end{array}
\end{equation} where $x_i \in {\mathbb{R}}^{N_i}$, $u_i \in
{\mathbb{R}}^{m_i}$ and the functions $f_i:
{\mathbb{R}}^{\sum_{j=1}^{n}N_j +
m_i}\rightarrow{\mathbb{R}}^{N_i}$ are continuous and for all $r
\in {\mathbb{R}}$ are locally Lipschitz continuous in
$x={({x_1}^T,\ldots,{x_n}^T)}^T$ uniformly in $u_i$ for $|u_i|\leq
r$. This regularity condition for $f_i$ guarantees the existence
and uniqueness of solution for the $i$th subsystem for a given
initial condition and input $u_i$.

The interconnection (\ref{subsystems}) can be written as
(\ref{whole system}) with $x:=(x_1^T,\dots,x_n^T)^T$,
$u:=(u_1^T,\dots,u_n^T)^T$  and
\[f(x,u)=\left(f_1(x_1,\dots,x_n,u_1)^T,\ldots,
f_n(x_1,\dots,x_n,u_n)^T\right)^T.\]
If we consider the individual subsystems, we treat the state $x_j,  j
\neq i$ as an independent input for the $i$th subsystem.

We now intend to formulate ISS conditions for the subsystems of
\eqref{subsystems}, where some conditions are in the sum formulation as in
\eqref{iss_sum} while other are given in the maximum form as in
\eqref{iss_max}.  Consider the index set $I:=\{1,\ldots,n\}$ partitioned into
two subsets $I_{\Sigma}$, $I_{\max}$ such that $I_{\max}=I \setminus
I_{\Sigma}$.

The $i$th subsystem of (\ref{subsystems}) is ISS, if
there exist functions
$\beta_i$ of class $\cal{KL}$, $\gamma_{ij}$, $\gamma_i \in
\cal{K}_\infty\cup\{$0$\}$ such that for all initial values
$x_i(0)$ and inputs $u \in L_{\infty}(\R_{+},\R^{m})$ there exists a unique
solution $x_i(\cdot)$ satisfying for all $t\ge0$
\begin{equation}\label{subssystems ISS_sum}
|x_i(t)|\leq
\beta_i(|x_i(0)|,t)+\sum\limits_{j=1}^{n}\gamma_{ij}({\|x_{j[0,t]}\|}_{\infty})+\gamma_i({\|u\|}_{\infty})
\,, \quad \text{if} \quad i\in I_{\Sigma} \,,
\end{equation}
 and
\begin{equation}\label{subssystems ISS_max}
\begin{array}{l}
|x_i(t)|
\leq\max\{\beta_i(|x_i(0)|,t),\max\limits_j\{\gamma_{ij}({\|x_{j[0,t]}\|}_{\infty})\},\gamma_i({\|u\|}_{\infty})\}
\,, \quad \text{if} \quad i\in I_{\max}\,.
\end{array}
\end{equation}
\begin{remark}\label{rem:index_subsets}
  Note that without loss of generality we can assume that
  $I_{\Sigma}=\{1,\ldots,k\}$ and $I_{\max}=\{k+1,\ldots,n\}$ where
  $k:=|I_{\Sigma}|$. This can be always achieved by a permutation of the
  subsystems in (\ref{subsystems}).
\end{remark}

Since ISS implies GS and the AG property, there exist functions
$\sigma_{i},\hat{\gamma}_{ij}$, $\hat{\gamma}_{i} \in
\cal{K}\cup\{$0$\}$, such that for any
 initial value $x_i(0)$ and
input $u \in L_{\infty}(\R_{+},\R^{m})$ there exists a unique
solution $x_i(t)$ and for all $t\geq 0$
\begin{eqnarray}
\label{subssystems GS_sum}
|x_i(t)| \leq&
\sigma_{i}(|x_i(0)|)+\sum\limits_{j=1}^{n}\hat{\gamma}_{ij}({\|x_{j[0,t]}\|}_{\infty})+\hat{\gamma}_{i}({\|u\|}_{\infty})\,,
& \quad \text{if} \quad i\in I_{\Sigma}\,,\\
\label{subssystems GS_max}
|x_i(t)|\leq &\max\{\sigma_{i}(|x_i(0)|),\max\limits_j\{\hat{\gamma}_{ij}({\|x_{j[0,t]}\|}_{\infty})\},\hat{\gamma}_{i}({\|u\|}_{\infty})\}\,,
& \quad \text{if} \quad i\in I_{\max}\,,
\end{eqnarray}
which are the defining inequalities for the GS property of the $i$-th
subsystem.

The AG property is defined in the same spirit by assuming that
there exist functions $\overline{\gamma}_{ij}$,
$\widetilde{\gamma}_i \in \cal{K}\cup\{$0$\}$, such that for any
initial value $x_i(0)$ and inputs $x_j \in
L_{\infty}(\R_{+},\R^{N_j})$, $i\neq j$, $u \in
L_{\infty}(\R_{+},\R^{m})$ there exists a unique solution $x_i(t)$
and
\begin{eqnarray}
\label{subssystems AG_sum} \limsup\limits_{t \to
\infty}|x_i(t)|\leq &
\sum\limits_{j=1}^{n}\overline{\gamma}_{ij}({\|x_{j}\|}_{\infty})+\overline{\gamma}_i({\|u\|}_{\infty})\,,
& \quad \text{if} \quad i\in I_{\Sigma}\,,\\
\label{subssystems AG_max}
\limsup\limits_{t \to \infty}|x_i(t)| \leq &
\max\{\max\limits_j\{\overline{\gamma}_{ij}({\|x_{j}\|}_{\infty})\},\overline{\gamma}_i({\|u\|}_{\infty})\}\,,
& \quad \text{if} \quad i\in I_{\max}\,.
\end{eqnarray}

We collect the gains $\gamma_{ij}\in \Kinf\cup \{0\}$ of the ISS conditions
\eqref{subssystems ISS_sum}, \eqref{subssystems ISS_max} in a matrix
$\Gamma=(\gamma_{ij})_{n\times n}$, with the convention $\gamma_{ii}\equiv
0$,\, $i=1,\dots,n$. The operator $\Gamma:\R_{+}^n\rightarrow\R_{+}^n$ is then
defined by
\begin{equation}\label{operator_gamma}
\Gamma(s):=\left(\Gamma_1(s),\ldots,\Gamma_n(s)\right)^T\,,
\end{equation}
where the functions $\Gamma_i:\R_{+}^n\rightarrow\R_{+}$ are given by
$\Gamma_i(s):=\gamma_{i1}(s_1)+\dots+\gamma_{in}(s_n)$ for $i\in
I_{\Sigma}$ and
$\Gamma_i(s):=\max\{\gamma_{i1}(s_1),\dots,\gamma_{in}(s_{n})\}$
for $i\in I_{\max}$. In particular, if
$I_{\Sigma}=\{1,\ldots,k\}$ and $I_{\max}=\{k+1,\ldots,n\}$ we
have
 \begin{equation}
\label{operator_gamma_reordered}
\Gamma(s)=\left(\begin{array}{c} \gamma_{12}(s_2)+\dots+\gamma_{1n}(s_n)\\
\vdots\\
\gamma_{k1}(s_1)+\dots+\gamma_{kn}(s_n)\\
\max\{\gamma_{k+1,1}(s_1),\dots,\gamma_{k+1,n}(s_n)\}\\
\vdots\\
\max\{\gamma_{n1}(s_1),\dots,\gamma_{n,n-1}(s_{n-1})\}
\end{array}\right)\,.
\end{equation}

In \cite{DRW07-mcss} small gain conditions were considered for the case
$I_{\Sigma}=I=\{1,\ldots,n\}$, respectively $I_{\max}=I$.
In \cite{Rueff07,DRW09} more general formulations of ISS are considered,
which encompass the case studied in this paper.  In this paper we exploit
the special structure to obtain more specific results than available before.

Our main question is whether the interconnection (\ref{subsystems}) is ISS
from $u$ to $x$. To motivate the approach we briefly recall the small gain
conditions for the cases $I_{\Sigma}=I$, resp. $I_{\max}=I$, which imply
ISS of the interconnection, \cite{DRW07-mcss}. If $I_{\Sigma}=I$, we need to
assume that there exists a $D:={\diag}_n(\id+\alpha)$, $\alpha \in
\K_{\infty}$ such that
\begin{equation}\label{small gain condition_sum}
\Gamma\circ D(s)\not\geq s,\, \forall s \in
{\R}_{+}^n\backslash\{0\}\,,
\end{equation}
 and if $I_{\max}=I$, then the small gain condition
\begin{equation}\label{small gain condition_max}
\Gamma(s)\not\geq s,\, \forall s \in {\R}_{+}^n\backslash\{0\}
\end{equation}
is sufficient. In case that both $I_{\Sigma}$ and $I_{\max}$ are not empty we
can use
\begin{equation}\label{max<sum}
\max_{i=1,\ldots,n}\{x_i\}\leq \sum_{i=1}^n x_i \leq n \max_{i=1,\ldots,n}
\{x_i\}
\end{equation}
to pass to the situation with $I_{\Sigma}=\emptyset$ or
$I_{\max}=\emptyset$. But this leads to more conservative gains. To avoid this
conservativeness we are going to obtain a new small gain condition for the case
$I_{\Sigma}\neq I \neq I_{\max}$. As we will see there are two essentially
equivalent approaches to do this. We may use the weak triangle inequality
\begin{equation}
  \label{eq:weaktri}
  a + b \leq \max \{ (\id + \eta) \circ a, (\id +\eta^{-1}) \circ b \}\,,
\end{equation}
which is valid for all functions $a,b, \eta \in \Kinf$ as discussed in
Section~\ref{sumtomax} to pass to a pure maximum formulation of ISS. However,
this method involves the right choice of a large number of weights in the weak
triangular inequality which can be a nontrivial problem.  Alternatively
tailor-made small gain conditions can be derived. The expressions in
(\ref{small gain condition_sum}), (\ref{small gain condition_max}) prompt us
to consider the following small gain condition. For a given $\alpha\in\Kinf$
let the diagonal operator $D_\alpha:\R_{+}^n\rightarrow\R_{+}^n$ be defined by
\begin{equation}
\label{D}
  D_\alpha(s):=(D_1(s_1),\ldots,D_n(s_n))^T\,,\quad s\in\R_{+}^n\,,
\end{equation}
where $D_i(s_i):=(\id+\alpha)(s_i)$ for $i\in I_{\Sigma}$
 and $D_i(s_i):=s_i$ for $i\in I_{\max}$.  The small gain
condition on the operator $\Gamma$ corresponding to a partition $I = I_\Sigma
\cup I_{\max}$ is then
\begin{equation}
\label{small gain condition_mix}
\exists\ \alpha\in \Kinf \quad : \quad
\Gamma\circ
D_\alpha(s)\not\geq s, \, \forall s \in {\R}_{+}^n\backslash\{0\}.
\end{equation}
We will abbreviate this condition as $ \Gamma\circ D_\alpha \not \geq
\id$.  In this paper we will prove that this small gain condition
guarantees the ISS property of the interconnection (\ref{subsystems}) and
show how an ISS-Lyapunov function can be constructed if this condition is
satisfied in the case of a Lyapunov formulation of ISS.

Before developing the theory we discuss an example to highlight the advantage of
the new small gain condition \eqref{small gain condition_mix},
cf. Theorem~\ref{theorem_iss}. In order not to cloud the issue we keep the
example as simple as possible.

\begin{example}
  We consider an interconnection of $n=3$ systems given by
\begin{equation}\label{eq:ex}
\begin{split}
\dot x_1=&-x_1+\gamma_{13}(|x_3|)+\gamma_1(u)\\
\dot x_2=&-x_2+\max\{\gamma_{21}(|x_1|),\gamma_{23}(|x_3|)\}\\
\dot x_3=&-x_3+\max\{\gamma_{32}(|x_2|),\gamma_3(u)\}
\end{split}
\end{equation}
where the $\gamma_{ij}$ are given $\Kinf$ functions. Using the variation of
constants method and the weak triangle inequality \eqref{eq:weaktri} we see
that the trajectories can be estimated by:
\begin{equation}\label{ex}
\begin{split}
|x_1(t)|&\leq\beta_1(|x(0)|,t)+\gamma_{13}(||x_{3[0,t]}||_{\infty})+\gamma_1({\|u\|}_{\infty})\\
|x_2(t)|&\leq \max\{\beta_2(|x(0)|,t),(\id +\eta)\circ\gamma_{21}(||x_{1[0,t]}||_{\infty}),
(\id +\eta)\circ\gamma_{23}(||x_{3[0,t]}||_{\infty})\}\\
|x_3(t)|&\leq \max\{ \beta_3(|x(0)|,t),(\id +\eta)\circ\gamma_{32}(||x_{2[0,t]}||_{\infty}),(\id +\eta)\circ\gamma_3({\|u\|}_{\infty})\}\,,
\end{split}
\end{equation}
where the $\beta_i$ are appropriate ${\cal KL}$ functions and $\eta \in \Kinf$
is arbitrary.

This shows that each subsystem is ISS.
In this case we have
\[\Gamma=\left(\begin{array}{ccc}
0&0&\gamma_{13}\\
(\id +\eta)\circ\gamma_{21}&0&(\id +\eta)\circ\gamma_{23}\\
0&(\id +\eta)\circ\gamma_{32}&0
\end{array}\right).
\]
Then the small gain condition (\ref{small gain condition_mix})
requires that there exists an $\alpha \in  \Kinf$ such that
\begin{equation}\label{ex:sgc}
\left(\begin{array}{c}
\gamma_{13}(s_3)\\
\max\{(\id +\eta)\circ\gamma_{21}\circ(\id+\alpha)(s_1),(\id +\eta)\circ\gamma_{23}(s_3)\}\\
(\id +\eta)\circ\gamma_{32}(s_2)
\end{array}\right)
\not\geq \left(\begin{array}{c}
s_1\\
s_2\\
s_3
\end{array}\right)\end{equation}
for all $s \in {\R}_{+}^3\backslash\{0\}$.  If (\ref{ex:sgc})
holds then considering $s^T(r):=(\gamma_{13}\circ(\id
+\eta)\circ\gamma_{32}(r),r,(\id +\eta)\circ\gamma_{32}(r))^T$,
$r>0$ we obtain that the following two inequalities are satisfied
\begin{eqnarray}
\label{ex_en1}
(\id+\alpha)\circ\gamma_{13}\circ(\id +\eta)\circ\gamma_{32}\circ(\id +\eta)\circ\gamma_{21}(r)<r,\\
\label{ex_en2} (\id +\eta)\circ\gamma_{23}\circ(\id
+\eta)\circ\gamma_{32}(r)<r.
\end{eqnarray}
It can be shown by contradiction that \eqref{ex_en1}
and \eqref{ex_en2} imply \eqref{ex:sgc}.

To give a simple example assume the that the gains are linear and given by
$\gamma_{13}:=\gamma_{21}:=\gamma_{23}:=\gamma_{32}(r)=0.9\, r$, $r\geq
0$. Choosing $\alpha=\eta = 1/10$ we see that the inequalities (\ref{ex_en1})
and \eqref{ex_en2}) are satisfied. So by Theorem \ref{theorem_iss} we conclude
that system (\ref{whole system}) is ISS. In this simple example we also see
that a transformation to the pure maximum case would have been equally
simple. An application of the weak triangle inequality for the first row with
$\eta = \alpha$ would have led to the pure maximization case. In this case the
small gain condition may be expressed as a cycle condition
\cite{Tee05,DRW07-mcss,Rueff07,JiW08,KaJ09}, which just yields the conditions
(\ref{ex_en1}) and \eqref{ex_en2}.

\end{example}
We would like to note that application of the small gain condition from
\cite{DRW07-mcss} will not help us to prove stability for this example, as can
be seen from the following example.
\begin{example}
  In order to apply results from \cite{DRW07-mcss} we
  could (e.g. by using \eqref{max<sum}) obtain estimates
  of the form
\begin{eqnarray}\label{ex_sum}
\nonumber
|x_1(t)|&\leq&\beta_1(|x(0)|,t)+\gamma_{13}(||x_{3[0,t]}||_{\infty})+\gamma_1({\|u\|}_{\infty})\\
|x_2(t)|&\leq&\beta_2(|x(0)|,t)+\gamma_{21}(||x_{1[0,t]}||_{\infty})+\gamma_{23}(||x_{3[0,t]}||_{\infty})\\
|x_3(t)|&\leq&\beta_3(|x(0)|,t)+\gamma_{32}(||x_{2[0,t]}||_{\infty})+\gamma_3({\|u\|}_{\infty}) \,.
\nonumber
\end{eqnarray}
With the gains from the previous example the corresponding gain matrix is
\begin{equation*}\Gamma =
  \begin{pmatrix}
    0 & 0 & 0.9\\ 0.9 & 0 & 0.9 \\ 0 & 0.9 & 0
  \end{pmatrix}\,,
\end{equation*}
and in the summation case with linear gains the small gain condition is
$r(\Gamma) <1$, \cite{DRW07-mcss}. In our example $r(\Gamma) > 1.19$, so that
using this criterion we cannot conclude ISS of the interconnection.
\end{example}

The previous examples motivate the use of the refined small gain condition
developed in this paper for the case of different ISS characterizations. In
the next section we study properties of the gain operators and show that mixed
ISS formulations can in theory always be transformed to a maximum formulation
without losing information on the small gain condition.

\section{Gain Operators}
\label{sec:Auxilary}

In this section we prove some auxiliary results for the operators satisfying
small gain condition (\ref{small gain condition_mix}). In particular, it will
be shown that a mixed (or pure sum) ISS condition can always be reformulated
as a maximum condition in such a way that the small gain property is
preserved.\footnote{We would like to thank one of the anonymous reviewers for
  posing the question whether this is possible.}

The
following lemma recalls a fact, that was already noted in
\cite{DRW07-mcss}.
\begin{lemma}
\label{lemma_gammacircd_dcirc_gamma}
For any $\alpha\in \Kinf$ the small gain condition $D_\alpha \circ \Gamma
\not \geq \id$ is equivalent to $ \Gamma  \circ D_\alpha
\not \geq \id$.
\end{lemma}
\begin{proof}
  Note that $D_\alpha$ is a homeomorphism with inverse $v\mapsto
  D^{-1}_\alpha(v):=\left(D_1^{-1}(v_1),\ldots,D_n^{-1}(v_n)\right)^T$.  By
  monotonicity of $D_\alpha$ and $D_\alpha^{-1}$ we have $D_\alpha \circ
  \Gamma(v)\not\geq v$ if and only if $\Gamma(v)\not\geq
  D_\alpha^{-1}(v)$. For any $w\in\R_{+}^n$ define $v=D_\alpha(w)$. Then $\Gamma
  \circ D_\alpha(w) \not\geq w$. This proves the equivalence.
\end{proof}

For convenience let us introduce $\mu: \R_{+}^n \times\R_{+}^n
\rightarrow \R_{+}^n $ defined by
\begin{equation}\label{eq:mu}
\mu(w,v):=\left(\mu_1(w_1,v_1),\ldots,\mu_n(w_n,v_n)\right)^T,w\in\R_{+}^n,\,
v\in\R_{+}^n,
\end{equation}
where $\mu_i:\R_{+}^2 \rightarrow \R_{+}$ is such that
$\mu_i(w_i,v_i):=w_i+v_i$ for $i\in I_{\Sigma}$ and
$\mu_i(w_i,v_i):=\max\{w_i,v_i\}$ for $i\in I_{\max}$. The following
counterpart of Lemma~13 in \cite{DRW07-mcss} provides the main technical
step in the proof of the main results.
\begin{lemma}\label{lemma_w_bound}

  Assume that there exists an $\alpha \in \Kinf$ such that the operator
  $\Gamma$ as defined in \eqref{operator_gamma} satisfies $\Gamma\circ
  D_\alpha \not\geq \id$ for a diagonal operator $D_\alpha$ as defined in
  \eqref{D}. Then there exists a $\phi \in \cal{K}_{\infty}$ such that for all
  $w, v \in {\mathbb{R}}_{+}^n$,
\begin{equation}\label{phi bound lemma}
w\leq\mu(\Gamma(w),v)
\end{equation}
implies $\|w\|\leq\phi(\|v\|)$.
\end{lemma}
\begin{proof}
Without loss of generality we assume $I_{\Sigma}=\{1,\dots,k\}$
and $I_{\max}=I\setminus I_{\Sigma}$, see
Remark~\ref{rem:index_subsets}, and hence $\Gamma$ is as in
(\ref{operator_gamma_reordered}).  Fix any $v \in
{\mathbb{R}}_{+}^n$. Note that for $v=0$ there is nothing to show, as
then $w\neq 0$ yields an immediate contradiction to the small gain
condition. So assume $v\neq 0$.

We first show, that for those $w \in {\mathbb{R}}_{+}^n$ satisfying
(\ref{phi bound lemma}) at least some components of $w$ have to be
bounded.  To this end let
$\widetilde{D}:{\mathbb{R}}_{+}^n\rightarrow{\mathbb{R}}_{+}^n$ be
defined by
\[\widetilde{D}(s):=\left(s_1+\alpha^{-1}(s_1),\dots,s_k+\alpha^{-1}(s_k),s_{k+1},\dots,s_n\right)^T,\, s\in\R_{+}^n\]
and let
$s^{\ast}:=\widetilde{D}(v)$.
Assume there exists $w=\left(w_1,\dots,w_n\right)^T$ satisfying
(\ref{phi bound lemma}) and such that $w_i>s^{\ast}_i$,
$i=1,\ldots,n$. In particular, for $i\in I_\Sigma$ we have
\begin{equation}\label{bounded_components_sum}
s^{\ast}_i < w_i  \leq
\gamma_{i1}(w_1)+\ldots+\gamma_{in}(w_n)+v_i
\end{equation}
and hence from the definition of $s^{\ast}$ it follows that
\[s^{\ast}_{i}=v_i+\alpha^{-1}(v_i) <
\gamma_{i1}(w_1)+\ldots+\gamma_{in}(w_n)+v_i.
\]
And so
$ v_i<\alpha(\gamma_{i1}(w_1)+\ldots+\gamma_{in}(w_n))$.
{}From (\ref{bounded_components_sum}) it follows
\begin{equation}\label{sum2}
\begin{array}{lll}
w_i\leq\gamma_{i1}(w_1)+\ldots+\gamma_{in}(w_n)+v_i<(id+\alpha)\circ(\gamma_{i1}(w_1)+\ldots+\gamma_{in}(w_n)).
\end{array}
\end{equation}
Similarly, by the construction of $w$ and the definition of $s^\ast$
we have for $i\in I_{\max}$
\begin{equation}\label{bounded_components_max}
v_i = s^{\ast}_{i}  <  w_{i} \leq
\max\{\gamma_{i1}(w_1),\ldots,\gamma_{in}(w_{n}),v_{i}\}\,,
\end{equation}
and hence
\begin{equation}\label{max2}
w_i\leq\max\{\gamma_{i1}(w_1),\ldots,\gamma_{in}(w_{n})\}.
\end{equation}
{}From (\ref{sum2}), (\ref{max2}) we get
$w\leq D_\alpha\circ\Gamma(w)$.  By Lemma~\ref{lemma_gammacircd_dcirc_gamma}
this contradicts the assumption $\Gamma\circ D_\alpha \not\geq \id$. Hence
some components of $w$ are bounded by the respective components of
$s^1:=s^{\ast}$. Iteratively we will prove that all components of $w$ are
bounded.

Fix a $w$ satisfying (\ref{phi bound lemma}). Then $w\not
>s^1$ and so there
exists an index set $I_1\subset I$, possibly depending
on $w$, such that $w_i>s_i^1$, $i\in I_1$ and $w_i\leq s_i^1$, for
$i \in I_1^c=I \setminus I_1$.
 Note that by the first step  $I_1^c$ is nonempty.
We now renumber the coordinates so that
 \begin{eqnarray}
\label{I_1 sum}
w_i>s_i^1\, \mbox{ and } & \,w_{i} \leq \sum\limits_{j=1}^{n}
\gamma_{ij}(w_j)+v_{i}\,, &\,i=1,\dots,k_1,\\
\label{I_1 max}
w_i>s_i^1\, \mbox{ and } & \,w_{i} \leq
\max\{\max\limits_{j}\gamma_{ij}(w_j),v_{i}\}\,,& \,i=k_1+1,\dots,n_1,
\\
\label{I_1^c sum}
w_i\leq s_i^1 \,\mbox{ and }& \,w_{i} \leq \sum\limits_{j=1}^{n}
\gamma_{ij}(w_j)+v_{i},\,, & i=n_1+1,\dots,n_1+k_2
\\
\label{I_1^c max}
w_i\leq s_i^1\, \mbox{ and } & \,w_{i} \leq
\max\{\max\limits_{j}\gamma_{ij}(w_j),v_{i}\}\,,& i=n_1+k_2+1,\dots,n \,,
\end{eqnarray}
where $n_1=|I_1|$, $k_1+k_2=k$. Using
(\ref{I_1^c sum}), (\ref{I_1^c max}) in (\ref{I_1 sum}), (\ref{I_1
max}) we get
\begin{eqnarray}
\label{I_1_sum}
w_{i} \leq&
\sum\limits_{j=1}^{n_1}\gamma_{ij}(w_j)+\sum\limits_{j=n_1+1}^{n}\gamma_{ij}(s_j^1)+v_{i},\
&
i=1,\dots,k_1,\\
\label{I_1_max}
w_{i} \leq&
\max\{\max\limits_{j=1,\dots,n_1}\gamma_{ij}(w_j),\max\limits_{j=n_1+1,\dots,n}\gamma_{ij}(s_j^1),v_{i}\},& i=k_1+1,\dots,n_1\,.
\end{eqnarray}
 Define $v^1\in \R_+^{n_1}$ by
 \begin{eqnarray*}
v_i^1:=&
\sum\limits_{j=n_1+1}^{n}\gamma_{ij}(s_j^1)+v_{i}\,,& i=1,\dots,k_1\,,\\
v_i^1:=&\max\{\max\limits_{j=n_1+1,\dots,n}\gamma_{ij}(s_j^1),v_{i}\}\,,&
i=k_1+1,\dots,n_1.
 \end{eqnarray*}
 Now (\ref{I_1_sum}), (\ref{I_1_max}) take the
form:
\begin{eqnarray}
\label{I_1_sum_v}
w_{i} \leq&
\sum\limits_{j=1}^{n_1}\gamma_{ij}(w_j)+v_i^1\,,& i=1,\dots,k_1,
\\
\label{I_1_max_v}
w_{i} \leq&
\max\{\max\limits_{j=1,\dots,n_1}\gamma_{ij}(w_j),v_i^1\}\,,&
i=k_1+1,\dots,n_1.
\end{eqnarray}
Let us represent $\Gamma=\left(\begin{array}{cc}
\Gamma_{I_1I_1}&\Gamma_{I_1I_1^c}\\\Gamma_{I_1^cI_1}&\Gamma_{I_1^cI_1^c}
\end{array}\right)$ and define the maps $\Gamma_{I_1I_1}:\R_{+}^{n_1}\rightarrow \R_{+}^{n_1}$,  $\Gamma_{I_1I_1^c}:\R_{+}^{n-n_1}\rightarrow \R_{+}^{n_1}$,
$\Gamma_{I^c_1I_1}:\R_{+}^{n_1}\rightarrow \R_{+}^{n-n_1}$ and
$\Gamma_{I_1^cI_1^c}:\R_{+}^{n-n_1}\rightarrow \R_{+}^{n-n_1}$
analogous to  $\Gamma$. Let
\[D_{I_1}(s):=\left((id+\alpha)(s_{1}),...\,,(id+\alpha)(s_{k_1}),s_{{k_1+1}},...\,,s_{{n_1}}\right)^T\,.\]
{}From
$\Gamma\circ D_\alpha(s) \not\geq s$ for all $s \neq 0$, $s \in
{\mathbb{R}}_{+}^n$ it follows by considering $s=(z^T,0)^T$ that
 $\Gamma_{I_1I_1}\circ D_{I_1}(z)
\not\geq z$ for all $z \neq 0$, $z \in {\mathbb{R}}_{+}^{n_1}$.
 Using the same approach as for $w\in\R_{+}^n$ it can be proved that
some components of $w^1=\left(w_1,\dots,w_{n_1}\right)^T$ are
bounded by the respective components of $s^2 := \widetilde{D}_{I_1}(v^1)$.

 We proceed inductively, defining
\begin{equation}\label{Ij}
I_{j+1} \subsetneqq I_j,\quad I_{j+1}:=\{i \in I_j:w_i>s_i^{j+1}\},
\end{equation}
with $I_{j+1}^c:=I \setminus I_{j+1}$ and
\begin{equation}\label{s_Ik}
s^{j+1}:=\widetilde{D}_{I_j}\circ
(\mu^j(\Gamma_{I_jI_j^c}(s_{I_j^c}^j),v_{I_j})),
\end{equation}
where $\widetilde{D}_{I_j}$ is defined analogously to $\widetilde{D}$, the map
$\Gamma_{I_jI_j^c}:\R_{+}^{n-n_j}\rightarrow\R_{+}^{n_j}$ acts analogously to
$\Gamma$ on vectors of the corresponding dimension,
$s_{I_j^c}^j=(s_i^j)_{i\in{I_j^c}}$ is the restriction defined in the
preliminaries and ${\mu}^j$ is appropriately defined similar to the
definition of $\mu$.\\
The nesting (\ref{Ij}), (\ref{s_Ik}) will end after at most $n-1$ steps: there
exists a maximal $l\leq n$, such that
$$I \supsetneqq I_1\supsetneqq \ldots \supsetneqq I_l\neq\emptyset$$
and all components of $w_{I_l}$ are bounded by the corresponding
components of $s^{l+1}$. Let
$$s_\varsigma:=\max\{s^{\ast},R_{I_1}(s^2),\ldots,R_{I_{l}}(s^{l+1})\}
:=\left(\begin{array}{c}
\max\{(s^{\ast})_1,(R_{I_1}(s^2))_1,\ldots,(R_{I_{l}}(s^{l+1}))_1\}\\
\vdots\\
\max\{(s^{\ast})_n,(R_{I_1}(s^2))_n,\ldots,(R_{I_{l}}(s^{l+1}))_n\}
\end{array}\right)$$
where $R_{I_j}$ denotes the anti-projection
${\mathbb{R}}_{+}^{|I_j|}\rightarrow{\mathbb{R}}_{+}^{n}$ defined above.

By the definition of $\mu$ for all $v\in\R_{+}^n$ it holds
\[0\leq v\leq \mu(\Gamma,\id)(v):=\mu(\Gamma(v),v).\] Let the $n$-fold
composition of a map $M:\R_+^n \to \R_+^n$ of the form $M\circ\ldots\circ M$
be denoted by ${[M]}^n$.  Applying $\widetilde{D}$ we have

\begin{equation}\label{estimates_for _v}\begin{split} 0
\leq
v\leq\widetilde{D}(v)\leq\widetilde{D}\circ&(\mu(\Gamma,\id))(v)\leq\dots\leq[\widetilde{D}\circ\mu(\Gamma,\id)]^n(v).
\end{split}
\end{equation}
{}From (\ref{s_Ik}) and (\ref{estimates_for _v}) for $w$ satisfying
(\ref{phi bound lemma}) we have
$w\leq s_\varsigma\leq[\widetilde{D}\circ\mu(\Gamma,\id)]^n(v)$.
The term on the right-hand side does not depend on any particular
choice of nesting of the index sets. Hence every $w$ satisfying
(\ref{phi bound lemma}) also satisfies
$w\leq{[\widetilde{D}\circ\mu(\Gamma,\id)]}^n({|v|}_{\max},\ldots,{|v|}_{\max})^T$
and taking the maximum-norm on both sides yields
${|w|}_{\max}\leq\phi({|v|}_{\max})$
for some function $\phi$ of class $\cal{K}_{\infty}$. For example,
$\phi$ can be chosen as

\[\phi(r):=\max\{({[\widetilde{D}\circ\mu(\Gamma,\id)]}^n(r,\ldots,r))_1,\ldots,({[\widetilde{D}\circ\mu(\Gamma,\id)]}^n(r,\ldots,r))_n\}.\]
This completes the proof of the lemma.\end{proof}

We also introduce the important notion of $\Omega$-paths \cite{DRW09}.  This
concept is useful in the construction of Lyapunov functions and will also be
instrumental in obtaining a better understanding of the relation between max
and sum small gain conditions.

\begin{definition}
A continuous path $\sigma\in\K_{\infty}^n$ is called an $\Omega$-path
with respect to $\Gamma$ if
\begin{enum}
\item for each $i$, the function $\sigma^{-1}_i$ is locally Lipschitz
  continuous on $(0,\infty)$;
\item for every compact set $P\subset(0,\infty)$ there are finite
constants $0<c<C$ such that for all points of differentiability of
$\sigma^{-1}_i$ and $i=1,\ldots,n$ we have
\begin{equation}\label{sigma cond 1}
0<c\leq(\sigma^{-1}_i)'(r)\leq C, \quad \forall r\in P
\end{equation}
\item for all $r>0$ it holds that $\Gamma(\sigma(r))<\sigma(r)$.
\end{enum}
\end{definition}

By \cite[Theorem~8.11]{DRW09} the existence of an $\Omega$-path $\sigma$
follows from the small gain condition \eqref{small gain condition_max}
provided an irreducibility condition is satisfied. To define this notion we
consider the directed graph $G({\cal V},{\cal E})$ corresponding to $\Gamma$
with nodes ${\cal V}=\{1,\ldots,n\}$. A pair $(i,j)\in {\cal V}\times{\cal V}$
is an edge in the graph if $\gamma_{ij} \neq 0$. Then $\Gamma$ is called
irreducible if the graph is strongly connected, see e.g. the appendix in
\cite{DRW07-mcss} for further discussions on this topic.

We note that if $\Gamma$ is reducible, then it may be brought into upper
block triangular form by a permutation of the indices
\begin{equation}\label{reducible_gamma}
{\Gamma}=\left(\begin{array}{cccc}
\Upsilon_{11}&\Upsilon_{12}&\ldots&\Upsilon_{1d}\\
0&\Upsilon_{22}&\ldots&\Upsilon_{2d}\\
\vdots&&\ddots&\vdots\\
0&\ldots&0&\Upsilon_{dd}
\end{array}
\right)
\end{equation}
where each block
$\Upsilon_{jj}\in(\K_{\infty}\cup\{0\})^{d_j\times d_j},
j=1,\ldots,d$, is either irreducible or 0.

The following is an immediate corollary to \cite[Theorem~8.11]{DRW09}, where
the result is only implicitly contained.
\begin{corollary}
  \label{c:Opath}
  Assume that $\Gamma$ defined in \eqref{operator_gamma} is
  irreducible. Then $\Gamma$ satisfies the small gain condition if and
  only if an $\Omega$-path $\sigma$ exists for $D\circ \Gamma$.
\end{corollary}

\begin{proof}
  The hard part is the implication that the small gain condition
  guarantees the existence of an $\Omega$-path, see \cite{DRW09}. For
  the converse direction assume that an $\Omega$-path exists for $D\circ
  \Gamma$ and that for a certain $s \in \R^n_+, s\neq 0$ we have $D\circ
  \Gamma(s) \geq s$. By continuity and unboundedness of $\sigma$ we may
  choose a $\tau>0$ such that $\sigma(\tau) \geq s$ but $\sigma(\tau)
  \not > s$. Then $s \leq D\circ \Gamma(s) \leq D\circ
  \Gamma(\sigma(\tau)) < \sigma(\tau)$. This contradiction proves the
  statement.
\end{proof}

\subsection{From Summation to Maximization}
\label{sumtomax}

We now use the previous consideration to show that an alternative approach
is possible for the treatment of the mixed ISS formulation, which
consists of transforming the complete formulation in a maximum
formulation. Using the weak triangle inequality
\eqref{eq:weaktri} iteratively the conditions in \eqref{subssystems ISS_sum}
may be transformed into conditions of the form \eqref{subssystems ISS_max}
with
\begin{eqnarray}\label{subssystems ISS_max-trans}
|x_i(t)| &\leq&
\beta_i(|x_i(0)|,t)+\sum\limits_{j=1}^{n}\gamma_{ij}({\|x_{j[0,t]}\|}_{\infty})+\gamma_i({\|u\|}_{\infty})\\
&\leq&\max\{\tilde{\beta}_i(|x_i(0)|,t),\max\limits_j\{\tilde{\gamma}_{ij}({\|x_{j[0,t]}\|}_{\infty})\},\tilde{\gamma}_i({\|u\|}_{\infty})\}
\label{subssystems ISS_max-trans2}
\end{eqnarray}
for $i\in I_{\Sigma}$. To get a general formulation we let
$j_1,\ldots,j_{k_i}$ denote the indices $j$ for which $\gamma_{ij}\neq
0$. Choose auxiliary functions
$\eta_{i0}, \ldots, \eta_{ik_i} \in \Kinf$ and
define $\chi_{i0} :=(\id +
\eta_{i0})$ and  
$\chi_{il} = (\id +
\eta_{i0}^{-1}) \circ \ldots \circ (\id + \eta_{i(l-1)}^{-1})\circ (\id +
\eta_{il})$, $l=1,\ldots, k_i$, and $\chi_{i(k_{i}+1)}=(\id +
\eta_{i0}^{-1})\circ \dots \circ (\id +
\eta_{ik_i}^{-1})$. Choose
a permutation $\pi_i:\{0,1,\ldots,k_i+1\} \to
\{0,1,\ldots, k_i+1\}$  and define
\begin{equation}
\label{eq:tildegains}
  \tilde{\beta}_i := \chi_{i\pi_i(0)}\circ \beta_i\,,\quad \tilde{\gamma}_{ij_l} :=
  \chi_{i{\pi_i(l)}}\circ \gamma_{ij_l}\,, \,\, l=1,\ldots,k_i
   \,,\quad \tilde{\gamma}_i:= \chi_{i\pi_i(k_i+1)}\circ \gamma_i  \,,
\end{equation}
and of course $\tilde{\gamma}_{ij} \equiv 0$,
$j\notin \{ j_1,\ldots,j_{k_1}\}$. In this manner the inequalities
\eqref{subssystems ISS_max-trans} are valid and a maximum ISS formulation is
obtained. Performing this for every $i\in I_\Sigma$ we obtain an operator
$\tilde{\Gamma} : \R_n^+ \to \R_n^+$ defined by
\begin{equation}
  \label{eq:tildeGamma}
  \left(\tilde\Gamma_1(s),\ldots,\tilde\Gamma_n(s)\right)^T\,,
\end{equation}
where the functions $\tilde\Gamma_i:\R_{+}^n\rightarrow\R_{+}$ are given by
$\tilde\Gamma_i(s):=\max\{
\tilde{\gamma}_{i1}(s_1),\dots,\tilde{\gamma}_{in}(s_n)\}$ for $i\in
I_{\Sigma}$ and
$\tilde\Gamma_i(s):=\max\{\gamma_{i1}(s_1),\dots,\gamma_{in}(s_{n})\}$ for $i\in
I_{\max}$. Here the $\tilde{\gamma}_{ij}$'s are given by
\eqref{eq:tildegains}, whereas the $\gamma_{ij}$'s are the original gains.

As it turns out the permutation is not really necessary and it is sufficient
to peel off the summands one after the other. We will now show that given a
gain operator $\Gamma$ with a mixed or pure sum formulation which
satisfies the small gain condition $D\circ \Gamma \not \geq \id$, it is always
possible to switch to a maximum formulation which also satisfies the
corresponding small gain condition $\tilde{\Gamma} \not \geq \id$.
%
%
In the following statement $k_i$ is to be understood as defined just after
\eqref{subssystems ISS_max-trans2}.
\begin{proposition}
  \label{mixedismax}
  Consider a gain operator $\Gamma$ of the form
  \eqref{operator_gamma}. Then the following two statements are equivalent
  \begin{enum}
  \item the small gain condition \eqref{small gain condition_mix} is satisfied,
  \item for each $i \in I_\Sigma$ there exist $\eta_{i,0}, \ldots,
    \eta_{i,(k_i+1)} \in \Kinf$, such that the corresponding small gain
    operator $\tilde{\Gamma}$ satisfies the small gain condition
    \eqref{small gain condition_max}.
  \end{enum}
\end{proposition}

\begin{remark}
  We note that in the case that a system \eqref{subsystems} satisfies a mixed
  ISS condition with operator $\Gamma$, then the construction in
  \eqref{subssystems ISS_max-trans} shows that the ISS condition is also
  satisfied in the maximum sense with the operator $\tilde\Gamma$. On the
  other hand the construction in the proof does not guarantee that if the ISS
  condition is satisfied for the operator $\tilde\Gamma$ then it will also be
  satisfied for the original $\Gamma$.
\end{remark}
\begin{proof}
  ``$\Rightarrow$'': We will show the statement under the condition that
  $\Gamma$ is irreducible. In the reducible case we may assume that $\Gamma$
  is in upper block triangular form \eqref{reducible_gamma}. In each of the
  diagonal blocks we can perform the transformation described below and the
  gains in the off-diagonal blocks are of no importance for the small gain
  condition.

  In the irreducible case we may apply Corollary~\ref{c:Opath} to obtain a
  continuous map $\sigma : [0,\infty) \to \R^n_+$, where $\sigma_i\in\Kinf$
  for every component function of $\sigma$ and so that
 \begin{equation}
 \label{path}
   D \circ \Gamma \circ \sigma (\tau) < \sigma(\tau) \,,\quad \text{for all
   }\quad \tau >0.
 \end{equation}
Define the homeomorphism $T:\R^n_+ \to \R^n_+$, $T:s \mapsto
(\sigma_1(s_1),\ldots, \sigma_n(s_n))$. Then $T^{-1} \circ D\circ \Gamma \circ
T \not\geq \id$ and we have by \eqref{path} for $e = \sum_{i=1}^n e_i$, that
\begin{equation*}
  T(\tau e) = \sigma(\tau) > D \circ \Gamma \circ \sigma(\tau) = D \circ
  \Gamma \circ T(\tau e)\,,
\end{equation*}
so that for all $\tau>0$
\begin{equation}
  \label{eq:newpath}
  T^{-1} \circ D \circ \Gamma \circ T (\tau e) < \tau e \,.
\end{equation}
We will show that $T^{-1} \circ \tilde{\Gamma} \circ T (\tau e) < \tau e$ for
an appropriate choice of the functions $\eta_{ij}$. By the converse direction
of Corollary~\ref{c:Opath} this shows that $T^{-1}
\circ \tilde{\Gamma} \circ T \not\geq \id $ and hence $\tilde{\Gamma} \not\geq
\id $ as desired.

Consider now a row corresponding to $i\in I_\Sigma$ and let $j_1,\ldots,
j_{k_i}$ be the indices for which $\gamma_{ij}\neq0$. For this row
\eqref{eq:newpath} implies
\begin{equation}
  \sigma_i^{-1} \circ (\id + \alpha) \circ \left( \sum_{j\neq i} \gamma_{ij}
  (\sigma_j(r)) \right) < r \,,\quad \forall r>0\,,
\end{equation}
or equivalently
\begin{equation}
  (\id + \alpha) \circ \left( \sum_{j\neq i} \gamma_{ij}
  \circ \sigma_j\circ \sigma_i^{-1} \right)  \circ \sigma_i(r) < \sigma_i(r)
  \,,\quad \forall r>0\,.
\end{equation}
This shows that
\begin{equation}
  (\id + \alpha) \circ \left( \sum_{j\neq i} \gamma_{ij}
   \circ \sigma_j\circ \sigma_i^{-1}  \right) < \id
  \,,\quad \text{on } (0,\infty)\,.
\end{equation}
Note that this implies that $\left(\id - \sum_{j\neq i} \gamma_{ij} \circ
  \sigma_j\circ \sigma_i^{-1}\right) \in \Kinf$ because $\alpha \in \Kinf$. We
may therefore choose $\hat{\gamma}_{ij} > \gamma_{ij} \circ \sigma_j\circ
\sigma_i^{-1}, j=j_1,\ldots,j_{k_i}$ in such a manner that
\begin{equation*}
  \id - \sum_{l=1}^{k_i} \hat{\gamma}_{ij_l}  \quad \in \Kinf \,.
\end{equation*}
Now define for $l=1,\ldots,k_i$
\begin{equation*}
  \eta_{il} := \left( \id - \sum_{k\leq l} \hat{\gamma}_{ij_k}
  \right)\circ \hat{\gamma}_{ij_l}^{-1} \quad \in \Kinf\,.
\end{equation*}
It is straightforward to check that
\begin{equation*}
  (\id + \eta_{il}) = \left( \id - \sum_{k< l} \hat{\gamma}_{ij_k}
  \right)\circ \hat{\gamma}_{ij_l}^{-1}\,,\quad
  (\id + \eta_{il}^{-1}) = \left( \id - \sum_{k< l} \hat{\gamma}_{ij_k}
  \right) \circ \left( \id - \sum_{k\leq l} \hat{\gamma}_{ij_k}
  \right)^{-1} \,.
\end{equation*}
With $\chi_{il}:= (\id + \eta_{i1}^{-1}) \circ \ldots \circ (\id +
\eta_{i({l-1})}^{-1}) \circ (\id + \eta_{il})$ it follows that
\begin{equation*}
  \chi_{il} \circ \gamma_{ij_l} \circ \sigma_{j_l}\circ
\sigma_i^{-1}= (\id + \eta_{i1}^{-1}) \circ \ldots \circ
(\id + \eta_{i,{l-1}}^{-1}) \circ (\id +
  \eta_{il}) \circ \gamma_{ij_l} \circ \sigma_{j_l}\circ
\sigma_i^{-1}= \hat{\gamma}_{ij_l}^{-1} \circ \gamma_{ij_l} \circ \sigma_{j_l}
\circ
\sigma_i^{-1}< \id\,.
\end{equation*}
This shows that it is possible to choose $\eta_{ij}, i\in I_\Sigma$ such that
all the entries in $T^{-1} \circ \tilde{\Gamma} \circ T$ are smaller than the
identity. This shows the assertion.

``$\Leftarrow$'': To show the converse direction let the small gain condition
\eqref{small gain condition_max} be satisfied for the operator
$\tilde{\Gamma}$. Consider $i\in I_{\Sigma}$.

We consider the following two cases for the permutation $\pi$ used in
\eqref{eq:tildegains}. Define $p:=\min\{\pi(0),\pi(k_i+1)\}$.  In the first
case $\{\pi(0),\pi(k_i+1)\} =\{k_i,k_i+1\}$, i.e., $ \pi(l)<p, \forall\
l\in\{1,\ldots,k_i\}$.
Alternatively, the second case is $\exists l\in\{1,\ldots,k_i\}: \pi(l)>p$.

We define $\alpha_i \in \K_\infty$ by
\begin{equation}\label{eq:alpha}
  \alpha_i:=\left\{\begin{array}{ll}
      \eta_{ip}^{-1}\circ\sum\limits_{\pi(l)>p}
       \gamma_{ij_l}\circ\left(\sum\limits_{j}
        \gamma_{ij}\right)^{-1}, & \mbox{ if } \exists j\in\{1,\ldots,k_i\}:
      \pi(j)>p\,, \\
      \eta_{i,p-1}\circ\gamma_{i,j_{\pi^{-1}(p-1)}}\circ\left(\sum\limits_{j}
        \gamma_{ij}\right)^{-1}, & \mbox { if } \forall j\in\{1,\ldots,k_i\}\quad
      \pi(j)<p \,.
 \end{array}
\right.
  \end{equation}
  Consider the $i$th row of $D\circ\Gamma$ and the case $\exists
  j\in\{1,\ldots,k_i\}: \pi(j)>p$. (Note that for no $l\in \{1,\ldots,k_i\}$
  we have $\pi(l) =p$).
\begin{equation}\label{1gamma_tilde_1}
\begin{array}{rcl}
  (\id+\alpha_i)\circ\sum\limits_{j} \gamma_{ij}&=&\sum\limits_{j} \gamma_{ij}+\alpha_i\circ\sum\limits_{j} \gamma_{ij}\\
  &=&\sum\limits_{j} \gamma_{ij}+\eta_{ip}^{-1}\circ\sum\limits_{
    \pi(l)>p}\gamma_{ij_l}\circ\left(\sum\limits_{j}
    \gamma_{ij}\right)^{-1}\circ\sum\limits_{j} \gamma_{ij}\\
  &=&\sum\limits_{j}
  \gamma_{ij}+\eta_{ip}^{-1}\circ\sum\limits_{\pi(l)>p}\gamma_{ij_l}\\
&=&\sum\limits_{\pi(l)<p}\gamma_{ij_l}+(\id+\eta_{ip}^{-1})\circ\sum\limits_{\pi(l)>p}\gamma_{ij_l} \,.
\end{array}
\end{equation}
Applying the weak triangle inequality \eqref{eq:weaktri} first to the
rightmost sum in the last line of \eqref{1gamma_tilde_1} and then to the
remaining sum we obtain
\begin{eqnarray}
\nonumber
&&\sum\limits_{\pi(l)<p}\gamma_{ij_l}+(\id+\eta_{ip}^{-1})\circ\sum\limits_{\pi(l)>p}\gamma_{ij_l}\\
\nonumber
&\leq&\sum\limits_{\pi(l)<p-1}\gamma_{ij_l}+\max\{(\id+\eta_{i,p-1})\circ\gamma_{i,\pi^{-1}(p-1)},\\
\nonumber
&&(\id+\eta_{i,p-1}^{-1})\circ(\id+\eta_{ip}^{-1})\circ
\max\limits_{\pi(l)>p}\{(\id+\eta_{i,{p+1}}^{-1})\circ\ldots\circ(\id+\eta_{i,\pi(l)-1}^{-1})
\circ(\id+\eta_{i\pi(l)})\circ\gamma_{ij_l}\}\}\\
&\leq&\ldots \quad
\leq \max\limits_{l}\{\chi_{i\pi(l)}\circ\gamma_{ij_l}\} \,.\label{1gamma_tilde_2}
\end{eqnarray}
The last expression is the defining equation for
$\widetilde{\Gamma}_i(s_1,\ldots,s_n) =
\max\limits_{l=1,\ldots,k_i}\{\chi_{i\pi(l)}\circ\gamma_{ij_l}(s_{j_l})\}$. Thus from
\eqref{1gamma_tilde_1}, \eqref{1gamma_tilde_2} we obtain
$\widetilde{\Gamma}_i\geq (D\circ\Gamma)_i$.

Consider now the case $\forall l\in\{1,\ldots,k_i\}\quad \pi(l)<p$.
A similar approach shows that $\widetilde{\Gamma}_i\geq (D\circ\Gamma)_i$.
Following the same steps as in the first case we obtain
\begin{eqnarray}\label{2gamma_tilde_1}
\nonumber
(\id+\alpha_i)\circ\sum\limits_{j} \gamma_{ij}&=
&\sum\limits_{j}
\gamma_{ij}+\eta_{i,p-1}\circ\gamma_{i,j_{\pi^{-1}(p-1)}}\\
\nonumber
&=&\sum\limits_{\pi(l)<p-1}\gamma_{ij_l}+(\id+\eta_{i,p-1})\circ\gamma_{i,j_{\pi^{-1}(p-1)}}\\
&\leq&\sum\limits_{\pi(l)<p-2}\gamma_{ij_l}+\max\{(\id+\eta_{i,p-2})\circ\gamma_{ij_{\pi^{-1}(p-2)}},\\
\nonumber
&&\phantom{\sum\limits_{\pi(l)<p-2}\gamma_{ij}+\max\{}(\id+\eta_{i,p-2}^{-1})\circ(\id+\eta_{i,(p-1)})\circ\gamma_{i,j_{\pi^{-1}(p-1)}}\}\\
\nonumber
&\leq &\ldots \quad \leq \max\limits_{l}\{\chi_{i\pi(l)}\circ\gamma_{ij_l}\} \,.
\end{eqnarray}

Again from \eqref{2gamma_tilde_1} $\widetilde{\Gamma}_i\geq (D\circ\Gamma)_i$.

Taking $\alpha=\min{\alpha_i}\in\K_{\infty}$ it holds that $\widetilde{\Gamma}\geq D\circ\Gamma$. Thus if $\widetilde{\Gamma}\not\geq\id$, then $D\circ\Gamma\not\geq\id$.
\end{proof}---------

\section{Small gain theorem}\label{sec:SGT}
We now turn back to the question of stability. In order to prove ISS of
(\ref{subsystems}) we use the same approach as in \cite{DRW07-mcss}. The main
idea is to prove that the interconnection is GS and AG and then to use the
result of \cite{SoW96} by which AG and GS systems are ISS.

So, let us first prove small gain theorems for GS and AG.

\begin{theorem}\label{theorem_gs}
Assume that each subsystem of (\ref{subsystems}) is GS
 and a gain matrix is given by $\Gamma=(\hat{\gamma}_{ij})_{n\times n}$. If there
exists $D$ as in (\ref{D}) such that $\Gamma\circ D(s) \not\geq s$
for all $s\neq 0,s\geq 0$ , then the system (\ref{whole
system}) is GS.
\end{theorem}
\begin{proof} Let us take the supremum over $\tau \in [0,t]$ on both sides of
(\ref{subssystems GS_sum}), (\ref{subssystems GS_max}). For $i\in
I_{\Sigma}$ we have
\begin{equation}\label{sup_gs_sum}
{\|x_{i[0,t]}\|}_{\infty}\leq
\sigma_{i}(|x_i(0)|)+\sum\limits_{j=1}^{n}\hat{\gamma}_{ij}({\|x_{j[0,t]}\|}_{\infty})+\hat{\gamma}_{i}({\|u\|}_{\infty})\end{equation}
and for $i\in I_{\max}$ it follows
\begin{equation}\label{sup_gs_max}
\begin{split} {\|x_{i[0,t]}\|}_{\infty} \leq
\max\{\sigma_{i}(|x_i&(0)|),\max\limits_j\{\hat{\gamma}_{ij}({\|x_{j[0,t]}\|}_{\infty})\},\hat{\gamma}_{i}({\|u\|}_{\infty})\}.
\end{split}
\end{equation}
Let us denote
$w=\left({\|x_{1[0,t]}\|}_{\infty},\dots,{\|x_{n[0,t]}\|}_{\infty}\right)^T$,
\[\begin{array}{lll}v&=&\left(\begin{array}{c}
\mu_1(\sigma_{1}(|x_1(0)|),\hat{\gamma}_{1}({\|u\|}_{\infty}))\\\vdots\\\mu_n(\sigma_{n}(|x_n(0)|),\hat{\gamma}_{n}({\|u\|}_{\infty}))\end{array}\right)=\mu(\sigma(|x(0)|),\hat{\gamma}({\|u\|}_{\infty})),\end{array}
\]
where we use notation $\mu$ and $\mu_i$ defined in \eqref{eq:mu}.
{}From (\ref{sup_gs_sum}), (\ref{sup_gs_max}) we obtain
$w\leq\mu(\Gamma(w),v)$. Then by Lemma~\ref{lemma_w_bound} there
exists $\phi\in\K_{\infty}$ such that
\begin{equation}\label{gs bound}
\begin{array}{lll}
{\|x_{[0,t]}\|}_{\infty}&\leq&
\phi(\|\mu(\sigma(|x(0)|),\hat{\gamma}({\|u\|}_{\infty}))\|)\\
&\leq&\phi(\|\sigma(|x(0)|)+\hat{\gamma}({\|u\|}_{\infty})\|)\\
&\leq&\phi(2\|\sigma(|x(0)|)\|)+\phi(2\|\hat{\gamma}({\|u\|}_{\infty})\|)
\end{array}
\end{equation}
for all $t>0$. Hence for
every initial condition and essentially bounded input $u$ the
solution of the system (\ref{whole system})  exists for all $t\geq
0$ and is uniformly bounded, since the right-hand side of (\ref{gs
bound}) does not depend on $t$. The estimate for GS is then given
by (\ref{gs bound}).\end{proof}
\begin{theorem}\label{theorem_ag}
Assume that each subsystem of (\ref{subsystems}) has the AG
property and that solutions of system (\ref{whole system}) exist for
all positive times and are uniformly bounded.
Let a gain matrix $\Gamma$ be given by $\Gamma=(\overline{\gamma}_{ij})_{n\times n}$.
If there exists a $D$ as in
(\ref{D}) such that $\Gamma\circ D(s) \not\geq s$ for all $s\neq
0, s\geq 0$, then system (\ref{whole system}) satisfies the AG property.
\end{theorem}
\begin{remark}
The existence of solutions for all times is essential, otherwise
the assertion is not true. See
Example~14 in \cite{DRW07-mcss}.
\end{remark}
\begin{proof}
Let $\tau$ be an arbitrary initial time. From the definition of
the AG property we have  for $i\in I_{\Sigma}$
\begin{equation}\label{ag_proof_sum}
\limsup\limits_{t \to \infty}|x_i(t)|\leq
\sum\limits_{j=1}^{n}\overline{\gamma}_{ij}({\|x_{j[\tau,\infty]}\|}_{\infty})+\overline{\gamma}_{i}({\|u\|}_{\infty})
\end{equation}
and for $i\in I_{\max}$
\begin{equation}\label{ag_proof_max}
\limsup\limits_{t \to \infty}|x_i(t)| \leq
\max\{\max\limits_j\{\overline{\gamma}_{ij}({\|x_{j[\tau,\infty]}\|}_{\infty})\}
,
\overline{\gamma}_{i}({\|u\|}_{\infty})\}.\end{equation}
Since all solutions of (\ref{subsystems}) are bounded
we obtain by \cite[Lemma~7]{DRW07-mcss} that
\[\limsup\limits_{t \to \infty}|x_i(t)|=\limsup\limits_{\tau \to
 \infty}(\|x_{i[\tau,\infty]}\|_{\infty})=:l_i(x_i),i=1,\ldots,n.\]
 By this property from (\ref{ag_proof_sum}),
(\ref{ag_proof_max}) and \cite[Lemma II.1]{SoW96} it follows that
\begin{equation}\notag
l_i(x_i)\leq
\sum\limits_{j=1}^{n}\gamma_{ij}(l_j(x_j))+\overline{\gamma}_{i}({\|u\|}_{\infty})
\end{equation} for $i\in I_{\Sigma}$
and
\begin{equation*}
l_i(x_i) \leq
\max\{\max\limits_j\{\gamma_{ij}(l_j(x_j))\},\overline{\gamma}_{i}({\|u\|}_{\infty})\}
\end{equation*} for $i\in I_{\max}$.
Using Lemma \ref{lemma_w_bound} we conclude
\begin{equation}\label{bound ag}
\limsup\limits_{t\to\infty}\|x(t)\|\leq\phi({\|u\|}_{\infty})
\end{equation}
for some $\phi$ of class $\cal{K}$, which is the desired AG
property.\end{proof}
\begin{theorem}\label{theorem_iss}
  Assume that each subsystem of (\ref{subsystems}) is ISS and let $\Gamma$ be
  defined by \eqref{operator_gamma}. If there exists a $D$ as in (\ref{D})
  such that $\Gamma\circ D(s) \not\geq s$ for all $s\neq 0,s\geq 0$, then
  system (\ref{whole system}) is ISS.
\end{theorem}
\begin{proof}
  Since each subsystem is ISS it follows in particular that it is GS with
  gains $\hat{\gamma}_{ij}\leq\gamma_{ij}$. By
  Theorem~\ref{theorem_gs} the whole interconnection (\ref{whole system}) is
  then GS. This implies that solutions of (\ref{whole system}) exists for all
  times.

Another consequence of ISS property of each subsystem is that each
of them has the AG property
with gains $\overline{\gamma}_{ij}\leq\gamma_{ij}$.
Applying Theorem~\ref{theorem_ag} the whole system (\ref{whole system}) has the AG property.

This implies that (\ref{whole system}) is ISS by Theorem~1 in \cite{SoW96}.
\end{proof}

\begin{remark}
  Note that applying Theorem~1 in \cite{SoW96} we lose information about the
  gains. As we will see in the second main result in
  Section~\ref{sec:Lyapunov} gains can be constructed in the framework of
  Lyapunov theory.
\end{remark}
\begin{remark}
  A more general formulation of ISS conditions for interconnected systems can
  be given in terms of so-called {\it monotone aggregation functions} (MAFs,
  introduced in \cite{Rueff07,DRW09}). In this general setting small
  gain conditions also involve a scaling operator $D$.  Since our construction
  relies on Lemma~\ref{lemma_w_bound} a generalization of the results in this
  paper could be obtained if sums are replaced by general MAFs and
  maximization is retained. We expect that the assertion of the
  Theorem~\ref{theorem_iss} remains valid in the more general case, at least
  if the MAFs are subadditive.
\end{remark}

The following section gives a Lyapunov type counterpart of the
small gain theorem obtained in this section and shows an explicit
construction of an ISS Lyapunov function for interconnections of
ISS systems.

\section{Construction of ISS Lyapunov functions}\label{sec:Lyapunov}
Again we consider an interconnection of $n$ subsystems in form of
\eqref{subsystems} where each subsystem is assumed to be ISS
and hence there is a smooth ISS Lyapunov function for each
subsystem. We will impose a small gain condition on the Lyapunov
gains to prove the ISS property of the whole system \eqref{whole
system} and we will look for an explicit construction of an ISS
Lyapunov function for it. For our purpose it is sufficient to work with not
necessarily smooth Lyapunov functions defined as follows.

A continuous function $\alpha: \R_+\rightarrow\R_+$, where $\alpha(r)=0$ if and only if $r=0$, is called positive definite.

A function $V:\R^n\rightarrow\R_+$ is called {\it proper and positive definite} if there are $\psi_1,\psi_2\in\K_{\infty}$ such that
\[\psi_1(\|x\|)\leq V(x)\leq\psi_2(\|x\|)\,, \quad \forall x\in\R^n.\]

\begin{definition}
A continuous function $V:\R^n\rightarrow\R_{+}$ is called an {\it
ISS Lyapunov function for the system} (\ref{whole system}) if

1) it is proper, positive definite and locally Lipschitz
continuous on $\R^n\backslash\{0\}$

2) there exists $\gamma\in\K$, and a positive definite function $\alpha$ such that in all points of differentiability of $V$ we have
\begin{equation}\label{lyap_func}
V(x)\geq\gamma(\|u\|)\Rightarrow \nabla V(x)f(x,u)\leq-\alpha(\|x\|).
\end{equation}
\end{definition}
Note that we do not require an ISS Lyapunov function to be smooth.
However as a locally Lipschitz continuous
function it is differentiable
almost everywhere.
\begin{remark}
In Theorem~2.3 in \cite{DRW09} it was proved that the system
(\ref{whole system}) is ISS if and only if it admits an (not
necessarily smooth) ISS Lyapunov function.
\end{remark}

ISS Lyapunov function for subsystems can be defined in the following way.

\begin{definition}
A continuous function $V_i:\R^{N_i}\rightarrow\R_{+}$ is called an
{\it ISS Lyapunov function for the subsystem} $i$ in
(\ref{subsystems}) if

1) it is proper and positive definite and locally Lipschitz continuous on $\R^{N_i}\backslash\{0\}$

2) there exist $\gamma_{ij}\in\K_{\infty}\cup\{0\}$, $j=1,\ldots,n$, $i\neq j$, $\gamma_i\in\K$ and a positive definite function $\alpha_i$ such that in all points of differentiability of $V_i$ we have

for $i\in I_{\Sigma}$
$$V_i(x_i)\geq\gamma_{i1}(V_1(x_1))+\ldots+\gamma_{in}(V_n(x_n))+\gamma_i(\|u\|)\Rightarrow$$
\begin{equation}\label{lyap_func_sum}
\nabla V_i(x_i)f_i(x,u)\leq-\alpha_i(\|x_i\|)
\end{equation}
and
for $i\in I_{\max}$
$$V_i(x_i)\geq\max\{\gamma_{i1}(V_1(x_1)),\ldots,\gamma_{in}(V_n(x_n)),\gamma_i(\|u\|)\}\Rightarrow$$
\begin{equation}\label{lyap_func_max}
\nabla V_i(x_i)f_i(x,u)\leq-\alpha_i(\|x_i\|).
\end{equation}
\end{definition}

Let the matrix $\overline{\Gamma}$ be obtained from matrix $\Gamma$ by
adding external gains $\gamma_i$ as the last column and let the
map $\overline{\Gamma}:\R_{+}^{n+1}\rightarrow\R_{+}^n$ be defined
by:
\begin{equation}\label{operator_gamma_bar}
\overline{\Gamma}(s,r):=\{\overline{\Gamma}_1(s,r),\dots,\overline{\Gamma}_n(s,r)\}
\end{equation}
for $s\in\R_{+}^n$ and $r\in\R_{+}$, where
$\overline{\Gamma}_i:\R_{+}^{n+1}\rightarrow\R_{+}$ is given by
$\overline{\Gamma}_i(s,r):=\gamma_{i1}(s_1)+\dots+\gamma_{in}(s_n)+\gamma_i(r)$
for $i\in I_{\Sigma}$ and by
$\overline{\Gamma}_i(s,r):=\max\{\gamma_{i1}(s_1),\dots,\gamma_{in}(s_{n}),\gamma_i(r)\}$
for $i\in I_{\Sigma}$.

Before we proceed to the main result of this section let us recall
a related result from \cite{DRW09} adapted to our situation:
\begin{theorem}\label{th:theorem5.3-fromSICON}
Consider the interconnection given by
\eqref{subsystems} where each subsystem $i$ has an ISS Lyapunov
function $V_i$ with the corresponding Lyapunov gains
$\gamma_{ij}$, $\gamma_i$, $i,j=1,\dots,n$ as in
\eqref{lyap_func_sum} and \eqref{lyap_func_max}. Let
$\overline{\Gamma}$ be defined as in \eqref{operator_gamma_bar}.
Assume that there is an $\Omega$-path $\sigma$ with respect to
$\Gamma$ and a function $\phi\in\Kinf$ such that
\begin{equation}\label{eq:5.3inSICON}
\overline\Gamma(\sigma(r),\phi(r))<\sigma(r),\quad\forall r>0.
\end{equation}
Then an ISS Lyapunov function for the overall system is given by
$$V(x)=\max_{i=1,\dots,n}\sigma_i^{-1}(V_i(x_i)).$$
\end{theorem}
We note that this theorem is a special case of \cite[Theorem~5.3]{DRW09} that
was stated for a more general $\overline\Gamma$ than here. Moreover it was
shown that an $\Omega$-path needed for the above construction always exists if
$\Gamma$ is irreducible and $\Gamma\not\ge\id$ in $\R^n_+$. The pure cases
$I_{\Sigma}=I$ and $I_{\max}=I$ are already treated in \cite{DRW09}, where the
existence of $\phi$ that makes Theorem~\ref{th:theorem5.3-fromSICON}
applicable was shown under the condition $D\circ\Gamma\not\ge\id$ for the case
$I_\Sigma=I$ and $\Gamma\not\ge\id$ for the case $I_{\max}=I$.

The next result gives a counterpart of \cite[Corollaries~5.5 and 5.6]{DRW09}
specified for the situation where both $I_{\Sigma}$ and $I_{\max}$ can be
nonempty.
\begin{theorem}\label{ISS irreducible}
  Assume that each subsystem of (\ref{subsystems}) has an ISS Lyapunov
  function $V_i$ and the corresponding gain matrix is given by
  (\ref{operator_gamma_bar}). If $\Gamma$ is irreducible and if there exists
  $D_\alpha$ as in (\ref{D}) such that $\Gamma\circ D_\alpha(s) \not\geq s$
  for all $s\neq 0, s\geq 0$ is satisfied, then the system (\ref{whole
    system}) is ISS and an ISS Lyapunov function is given by
\begin{equation}\label{lyap_func_explicit}
V(x)=\max\limits_{i=1,\ldots,n}\sigma_i^{-1}(V_i(x_i)),
\end{equation}
where $\sigma\in\K_{\infty}^n$ is an arbitrary $\Omega$-path with respect to $D\circ\Gamma$.
\end{theorem}
\begin{proof}
{}From the structure of $D_\alpha$ it follows that
\[\begin{array}{llll}
\sigma_i&>&(\id+\alpha)\circ\Gamma_i(\sigma), &\quad i\in I_{\Sigma},\\
\sigma_i&>&\Gamma_i(\sigma),& \quad i\in I_{\max}.
\end{array}\]
The irreducibility of $\Gamma$ ensures that $\Gamma(\sigma)$ is
unbounded in all components. Let $\phi\in\K_{\infty}$ be such that
for all $r\geq 0$ the inequality
$\alpha(\Gamma_i(\sigma(r)))\geq\max\limits_{i=1,
  \ldots,n}\gamma_{i}(\phi(r))$ holds for $i\in I_{\Sigma}$ and
$\Gamma_i(\sigma(r))\geq\max\limits_{i=1,
  \ldots,n}\gamma_{i}(\phi(r))$ for $i\in I_{\max}$.
Note that such a $\phi$ always exists and can be chosen as follows.
For any $\gamma_i\in\K$ we choose
$\tilde\gamma_i\in\Kinf$ such that $\tilde\gamma_i\ge\gamma_i$.
Then $\phi$ can be taken as
$\phi(r):=\frac{1}{2}\min\{\min\limits_{i\in I_{\Sigma},j\in
  I}\tilde\gamma_j^{-1}(\alpha(\Gamma_i(\sigma(r)))),\min\limits_{i\in
  I_{\max},j\in I}\tilde\gamma_j^{-1}(\Gamma_i(\sigma(r)))\}$. Note
that $\phi$ is a $\Kinf$ function since the minimum over $\Kinf$
functions is again of class $\Kinf$. Then we have for all
$r>0,i\in I_{\Sigma}$ that
\[\sigma_i(r)>D_i\circ\Gamma_i(\sigma(r))=\Gamma_i(\sigma(r))+\alpha(\Gamma_i(\sigma(r)))\geq\Gamma_i(\sigma(r))+\gamma_{i}(\phi(r))=\overline{\Gamma}_i(\sigma(r),\phi(r))\]
and for all $r>0,i\in I_{\max}$
\[\sigma_i(r)>D_i\circ\Gamma_i(\sigma(r))=\Gamma_i(\sigma(r))\geq\max\{\Gamma_i(\sigma(r)),\gamma_{i}(\phi(r))\}=\overline{\Gamma}_i(\sigma(r),\phi(r)).\]
Thus $\sigma(r)>\overline{\Gamma}(\sigma(r),\phi(r))$ and the
assertion follows from Theorem~\ref{th:theorem5.3-fromSICON}.
\end{proof}

The irreducibility assumption on $\Gamma$ means in particular that the graph
representing the interconnection structure of the whole system is strongly
connected. To treat the reducible case we consider an approach using the
irreducible components of $\Gamma$. If a matrix is reducible it can be
transformed to an upper block triangular form via a permutation of the
indices, \cite{BeP79}.

The following result is based on
\cite[Corollaries~6.3 and 6.4]{DRW09}.

\begin{theorem}
  Assume that each subsystem of (\ref{subsystems}) has an ISS Lyapunov
  function $V_i$ and the corresponding gain matrix is given by
  (\ref{operator_gamma_bar}). If there exists $D_\alpha$ as in (\ref{D}) such
  that $\Gamma\circ D_\alpha(s) \not\geq s$ for all $s\neq 0, s\geq 0$ is
  satisfied, then the system (\ref{whole system}) is ISS, moreover there
  exists an $\Omega$-path $\sigma$ and $\phi\in\Kinf$ satisfying
  $\overline\Gamma(\sigma(r),\phi(r))<\sigma(r), \forall \ r>0$ and an ISS
  Lyapunov function for the whole system \eqref{whole system} is given by
\[V(x)=\max\limits_{i=1,\ldots,n}\sigma_i^{-1}(V_i(x_i)).\]
\end{theorem}
\begin{proof}
After a renumbering of subsystems we can assume that
$\overline{\Gamma}$ is of the form (\ref{reducible_gamma}).
Let $D$ be the corresponding diagonal operator that contains $\id$
or $\id+\alpha$ on the diagonal depending on the new enumeration of
the subsystems. Let the state $x$ be partitioned into
$z_i\in\R^{d_i}$ where $d_i$ is the size of the $i$th diagonal
block $\Upsilon_{ii}$, $i=1,\dots,d$. And consider the subsystems
$\Sigma_j$ of the whole system \eqref{whole system} with these
states
\[z_j:=(x_{q_j+1}^T,x_{q_j+2}^T,\dots,x_{q_{j+1}}^T)^T,\] where
$q_j=\sum_{l=1}^{j-1}d_l$, with the convention that $q_1=0$. So the subsystems
$\Sigma_j$ correspond exactly to the strongly connected components of the
interconnection graph.  Note that each $\Upsilon_{jj},j=1,\dots,d$ satisfies a
small gain condition of the form $\Upsilon_{jj}\circ D_j\not\ge\id$ where
$D_j:\R^{d_j}\rightarrow\R^{d_j}$ is the corresponding part of $D_\alpha$.

For each $\Sigma_j$ with the gain operator $\Upsilon_{jj},j=1,\dots,d$ and
external inputs $z_{j+1},\dots,z_d, u$ Theorem~\ref{ISS irreducible} implies
that there is an ISS Lyapunov function
$W_j=\max\limits_{i=q_j+1,\dots,q_{j+1}}\widehat{\sigma}_i^{-1}(V_i(x_i))$ for
$\Sigma_j$, where
$(\widehat{\sigma}_{q_j+1},\dots,\widehat{\sigma}_{q_{j+1}})^T$ is an
arbitrary $\Omega$-path with respect to $\Upsilon_{jj}\circ D_j$. We will show
by induction over the number of blocks that an ISS Lyapunov function for the
whole system \eqref{whole system} of the form
$V(x)=\max\limits_{i=1,\ldots,n}\sigma_i^{-1}(V_i(x_i))$ exists, for an
appropriate $\sigma$.

For one irreducible bock there is nothing to show. Assume that for the system
corresponding to the first $k-1$ blocks an ISS Lyapunov function exists and is
given by
$\widetilde{V}_{k-1}=\max\limits_{i=1,\ldots,q_{k}}\sigma_i^{-1}(V_i(x_i))$.
Consider now the first $k$ blocks with state $(\widetilde{z}_{k-1}, z_k)$, where
$\widetilde{z}_{k-1}:=(z_1,\dots,z_{k-1})^T$. Then we have the implication
\begin{eqnarray*}
  \widetilde{V}_{k-1}(\widetilde{z}_{k-1})\ \geq \
  \widetilde{\gamma}_{k-1,k}(W_k(z_k))+
  \widetilde{\gamma}_{k-1,u}(\|u\|)\quad \Rightarrow
  \\
\nabla\widetilde{V}_{k-1}(\widetilde{z}_{k-1})
 \widetilde{f}_{k-1}(\widetilde{z}_{k-1},z_k,u)\ \leq \
 -\widetilde{\alpha}_{k-1}(\|\widetilde{z}_{k-1}\|)\,,
\end{eqnarray*}
where $\widetilde{\gamma}_{k-1,k}, \widetilde{\gamma}_{k-1,u}$ are the
corresponding gains, $\widetilde{f}_{k-1}$, $\widetilde{\alpha}_{k-1}$ are the
right hand side and dissipation rate of the first $k-1$ blocks.

The gain matrix corresponding to the block $k$ then has the form
\[\overline{\Gamma}_k=\left(\begin{array}{ccc}
0&\widetilde{\gamma}_{k-1,k}&\widetilde{\gamma}_{k-1,u}\\
0&0&\gamma_{k,u}
\end{array}
\right).\]

For $\overline{\Gamma}_k$ by \cite[Lemma~6.1]{DRW09} there
exist an $\Omega$-path
$\widetilde{\sigma}^k=(\widetilde{\sigma}_1^k,\widetilde{\sigma}_2^k)^T\in
\K_{\infty}^2$ and $\phi\in\K_{\infty}$ such that
$\overline\Gamma_k(\widetilde{\sigma}^k,\phi)<\widetilde{\sigma}^k$
holds.  Applying Theorem~\ref{th:theorem5.3-fromSICON} an ISS
Lyapunov function for the whole system exists and is given by
\begin{eqnarray}\notag
\widetilde{V}_k&=&\max\{(\widetilde{\sigma}_1^k)^{-1}(\widetilde{V}_{k-1}),(\widetilde{\sigma}_2^k)^{-1}(W_k)\}
\end{eqnarray}

A simple inductive argument shows that the final Lyapunov function is of the
form $V(x)=\max\limits_{k=1,\ldots,d}(\sigma_k^{-1}(W_k(z_k))$, where for
$k=1,\ldots, d-1$ we have (setting $\sigma_2^0=\id$)
\[ \sigma_k^{-1} = \left( \tilde{\sigma}_1^{d-1} \right)^{-1} \circ \dots \circ
\left( \tilde{\sigma}_1^{k} \right)^{-1} \circ
\left( \tilde{\sigma}_2^{k-1} \right)^{-1} \]
and $\sigma_d = \tilde{\sigma}_2^{d-1}$.
This completes the proof.
\end{proof}

\section{Conclusion}\label{sec:Conclusion}
We have considered large-scale interconnections of ISS systems. The mutual
influence of the subsystems on each other may either be expressed in terms of
summation or maximization of the corresponding gains. We have shown that such
a formulation may always be reduced to a pure maximization formulation,
however the presented procedure requires the knowledge of an $\Omega$-path of
the gain matrix, which amounts to having solved the problem. Also an
equivalent small gain condition has been derived which is adapted to the
particular problem. A simple example shows the effectiveness and advantage of
this condition in comparison to known results. Furthermore, the Lyapunov
version of the small gain theorem provides an explicit construction of ISS
Lyapunov function for the interconnection.
\section*{Acknowledgment}

The authors would like to thank the anonymous reviewers for careful reading
and helpful comments and in particular for pointing out the
question that lead to the result in Proposition~\ref{mixedismax}.

This research is funded by the Volkswagen Foundation (Project Nr.  I/82684
"Dynamic Large-Scale Logistics Networks"). S. Dashkovskiy is funded by the DFG
as a part of Collaborative Research Center 637 "Autonomous Cooperating
Logistic Processes - A Paradigm Shift and its Limitations".

\end{document}